\documentclass[
final
]{dmtcs-episciences}

\usepackage[utf8]{inputenc}
\usepackage{subfigure}
\usepackage{amssymb}
\usepackage{amsmath,amsthm}

\usepackage[round]{natbib}

\newtheorem{remark}{Remark}

 \newtheorem{theorem}[remark]{Theorem}
 \newtheorem{proposition}[remark]{Proposition}
 \newtheorem{corollary}[remark]{Corollary}

\author{Dorota Kuziak\affiliationmark{1}
  \and Iztok Peterin\affiliationmark{2,3}\thanks{Partially supported by the Ministry of Science of Slovenia under the grant P1-0297.}
  \and Ismael G. Yero\affiliationmark{4}}
\title[Open $k$-monopolies in graphs: complexity and related concepts]{Open $k$-monopolies in graphs: complexity and related concepts}
\affiliation{
  Departament d'Enginyeria Inform\`atica i Matem\`atiques, Universitat Rovira i Virgili, Spain\\
  University of Maribor, Faculty of Electrical Engineering and Computer Science, Slovenia\\
  Institute of Mathematics, Physics and Mechanics, Slovenia\\
	Escuela Polit\'ecnica Superior, Universidad de C\'adiz, Spain}
\received{2015-06-12}
\revised{2015-11-17}
\accepted{2016-03-07}
\begin{document}
\publicationdetails{18}{2016}{3}{6}{1407}
\maketitle
\begin{abstract}
  Closed monopolies in graphs have a quite long range of applications in several problems related to overcoming
failures, since they frequently have some common approaches around the notion of majorities, for instance
to consensus problems, diagnosis problems or voting systems. We introduce here open $k$-monopolies in graphs which
are closely related to different parameters in graphs. Given a graph $G=(V,E)$ and $X\subseteq V$, if $\delta_X(v)$
is the number of neighbors $v$ has in $X$, $k$ is an integer and $t$ is a positive integer, then we
establish in this article a connection between the following three concepts:
\begin{itemize}
	\item Given a nonempty set $M\subseteq V$ a vertex $v$ of $G$ is said to be $k$-controlled by $M$
	if $\delta_M(v)\ge \frac{\delta_V(v)}{2}+k$. The set $M$ is called an open $k$-monopoly for $G$ if it
	$k$-controls every vertex $v$ of $G$.
	\item  A function $f: V\rightarrow \{-1,1\}$ is called a signed total $t$-dominating function
	 for $G$ if $f(N(v))=\sum_{v\in N(v)}f(v)\geq t$ for all $v\in V$.
	\item A nonempty set $S\subseteq V$ is a global (defensive and offensive) $k$-alliance in $G$ if
	$\delta_S(v)\ge \delta_{V-S}(v)+k$ holds for every $v\in V$.
\end{itemize}
In this article we prove that the problem of computing the minimum cardinality of an open $0$-monopoly in
a graph is NP-complete even restricted to bipartite or chordal graphs. In addition we present some general bounds for the minimum cardinality
of open $k$-monopolies and we derive some exact values.
\end{abstract}

\section{Introduction}

We begin stating some terminology and notation which we will use. Throughout this article, $G$ denotes
a simple graph with vertex set $V(G)$ and edge set $E(G)$ (we will use only $V$ and $E$ if the graph is
clear from the context). The order of $G$ is $n = |V(G)|$ and the size is $m=|E(G)|$. We denote two
adjacent vertices $u$ and $v$ by $u \sim v$. Given a vertex $v\in V,$ the set $N(v)=\{u\in V: \; u\sim v\}$
is the {\em open neighborhood} of $v$, and  the set  $N[v]=N(v)\cup \{v\}$ is the {\em closed neighborhood}
of $v$. So, the  {\em degree} of a vertex $v \in V$ is  $\delta(v)=|N(v)|$. Given a set $S\subset V$, the
{\em open neighborhood} of $S$ is $N(S)=\bigcup_{v\in S} N(v)$ and the {\em closed neighborhood} of $S$
is $N[S]=N(S)\cup S$. The minimum and maximum degree of $G$ are denoted by $\delta(G)$ and $\Delta(G)$,
respectively (again we use $\delta$ and $\Delta$ for short if $G$ is clear from the context). For a nonempty
set $S \subseteq V$ and a vertex $v \in V$, $N_S(v)$ denotes the set of neighbors $v$ has in $S$, \emph{i.e.},
$N_S(v) = S\cap N(v)$. The degree of $v$ in $S$ will be denoted by $\delta_S(v) = |N_S(v)|$. Also,
$\overline{S}=V-S$ is the complement of a set $S$ in $V$ and $\partial S=N[S]-S$ is the boundary of a set
$S$. The subgraph of $G$ induced by a set $S$ is denoted by $\langle S\rangle$.

In the first article, see \cite{monopoly-one}, on closed monopolies in graphs (called monopolies there)
the following terminology was used. A vertex $v$ in $G$ is said to be controlled by a set $M\subset V$ if at least half of its
closed neighborhood is in $M$. The set $M$ is called a {\em closed monopoly} if it controls every vertex $v$ of $G$.
Equivalently, the set $M$ is a closed monopoly in $G$, if for any vertex $v\in V(G)$ it
follows that $|N[v]\cap M|\ge \left\lceil\frac{|N[v]|}{2}\right\rceil$. In this article, we introduce
open $k$-monopolies in a natural way, by replacing closed neighborhoods with open neighborhoods. Hence, we can use the degree
of vertices instead of cardinalities of closed neighborhoods. Given some integer $k$, a vertex $v$ of $G$ is said to be $k$-controlled
by a set $M$ if $\delta_M(v)\ge \frac{\delta(v)}{2}+k$. Analogously, the set $M$ is called  an {\em open} $k$-{\em monopoly}
if it $k$-controls every vertex $v$ of $G$. Notice that not for every value of $k$
there exists an open $k$-monopoly in $G$ (further on we give some suitable interval for such $k$). Also, note that, close and open
monopolies cannot be exactly compared, since in a closed monopoly a vertex $v$ also counts itself in controlling $v$, which is not the case in
any open monopoly. The smallest example is already $K_2$, where is only one vertex in a closed monopoly, but both vertices are necessary in an open $0$-monopoly.
Differently, there are only two vertices in a minimum open $0$-monopoly of $P_5$, while we need at least three vertices in every closed
monopoly of $P_5$. In this article, we are focused only in open $k$-monopolies. In this sense, from now on we omit the term ``open'' and just use the terminology of $k$-monopolies. On the other hand, we remain using the term closed monopoly whenever referring to some previous work on this topic.

According to \cite{small-coalition}, several problems related to overcoming failures have some common
approaches around the notion of majorities. Their ideas are directed toward decreasing, as much as possible, the
damage caused due to failed vertices; by maintaining copies of the most important data and performing a voting
process among the participating processors in situation that failures occur; and by adopting as true those data
stored at the majority of the not failed processors. This idea is also commonly used in some fault tolerant algorithms
including agreement and consensus problems (see \cite{consensus}), diagnosis problems (see \cite{diagnosis}) or voting systems
(see \cite{voting}), among other applications and references.

\cite{small-coalition} were interested into locality based on the
following facts. Frequently, processors running in a system are better aware of whatever happens in their neighborhood
than outside of it. Moreover, some distributed network models allow only for computations developed with local processors,
which means that, a processor can only obtain a data from other processors having a ``relative'' close distance from itself.
Therefore, it is more efficient to store data as locally as possible.

Nevertheless, there could exists also a risk in this way. If the voting is restricted to local neighborhoods, we could produce
a sufficiently large set of failures which will probably constitute the majority in some of
these neighborhoods. In this sense, see \cite{small-coalition}, the authors assert the following:
{\em once the voting is performed over subsets of vertices, the ability of failed vertices to influence the outcome of the votes
becomes not only a function of their number but also a function of their location in the network: well situated vertices can
acquire greater influence}. This simple fact led them to study the problem of characterizing the potential power of a set of
failures in a network of processors, and as a consequence, the study of (closed) monopolies in graphs.

Notions of closed monopolies in graphs were introduced first by \cite{monopoly-one}, where
several ideas regarding voting systems were described. Once such article appeared, a high number of researches were
devoted to such parameter and its relationship with other similar structures like (defensive and offensive) alliances, see
\cite{alliancesOne}, or signed dominating functions, see \cite{signed-one}, among other works. An interesting article,
where several of these connections are dealt with, is from \cite{fernau-survey}. Moreover, this article presents a possible
generalization of all these  (closed) monopolies-related structures which comprise them altogether. The complexity of closed
monopolies in graphs is also well studied. The NP-hardness of finding the minimum cardinality of a closed monopoly in a graph
is easy to observe as stated by \cite{monopoly-one}. In such work was also pointed out a conjecture concerning the inapproximability
of such problem. A weaker version of such conjecture has been proved by \cite{monopol6}. In addition some other
inapproximability results of this problem have appeared by \cite{mishra} and \cite{monopol2}. Particularly, in \cite{monopol2},
these results are centered in regular graphs. Moreover, there it is also proved that for the case of tree graphs,
a closed monopoly of minimum cardinality can be computed in linear time. On the other hand, see \cite{khoskah}, some relationships
and bounds for the minimum cardinality of closed monopolies in graphs are stated in terms of matchings and/or girths. Also, dynamic
closed monopolies has been introduced in connection with modeling some problems of spreading the influence in social networks
(see \cite{small-coalition,monopol3}). Other studies in dynamic closed monopolies can be found in \cite{monopol1} and in \cite{monopol4}.

\section{Concepts related to monopolies}

Many times mathematical concepts are defined independently in two or even more papers.
When this occurs, the equivalence sometimes is obvious (mostly when papers occur in the same
time period), but sometimes we need more effort to find the connection (mostly when there is
a longer time period between publications). This may yield not
sufficient effort of later authors with the history, but we rather present it as an enough important
concept to start to investigate it from different point of view.

The above holds (at least partial) for signed (total) domination (introduced first by \cite{heni} (by \cite{signed-total})) and for different
types of alliances (introduced first by \cite{alliancesOne}). We add monopolies to this list
and present these connections in this section.

\subsection{Alliances}

Alliances in graphs were introduced first by \cite{alliancesOne} and
generalized to $k$-alliances by \cite{kdaf,kdaf1}. After that several
works have been developed in this topic. Remarkable cases are
\cite{fava} and \cite{GlobalalliancesOne}. Relationships with different parameters
of the graphs have been obtained and the alliances of several families
of graphs have been studied. A nonempty set $S\subseteq V$ is a
\textit{defensive $k$-alliance} in $G$ for $k\in\{-\Delta, \ldots , \Delta\}$
if for every $v\in S$
\begin{equation}\label{cond-defensiva-grado2}
\delta_S(v)\geq \delta_{\overline{S}}(v)+k.
\end{equation}
Moreover,  for $k\in\{2-\Delta, \ldots , \Delta\}$, a nonempty set
$S \subseteq V$ is an \textit{offensive $k$-alliance} in $G$ if for
every $v\in \partial S$
\begin{equation}\label{cond-offensiva-grado1}
\delta_S(v)\geq \delta_{\overline{S}}(v)+k.
\end{equation}
A nonempty set $S\subseteq V$ is a \textit{powerful $k$-alliance} if
$S$ is a defensive $k$-alliance and an offensive ($k+2$)-alliance. A set $D$ is a \emph{dominating set} in $G$
if every vertex outside of $D$ is adjacent to at least one vertex of $G$.

A (defensive, offensive or powerful) $k$-alliance is called \emph{global}
if it is a dominating set. The \emph{global}
\textit{defensive $($offensive$)$ $k$-alliance number} of $G$, denoted
by $\gamma_k^d(G)$ ($\gamma_k^{o}(G)$), is defined as the minimum cardinality
of a global defensive (offensive) $k$-alliance in $G$. For
$k\in\{-\Delta,...,\Delta-2\}$, the  \emph{global} \textit{powerful $k$-alliance number}
of $G$, denoted by $\gamma_k^{p}(G)$, is defined as the minimum cardinality
of a global powerful $k$-alliance in $G$. A global powerful alliance of minimum
cardinality in $G$ is called a $\gamma_k^{p}(G)$-set of $G$. Notice that there
exist graphs not containing any global powerful $k$-alliance for some specific
values of $k$. In this sense, in this work we are interested in those graphs
having global powerful $k$-alliances. It means that whenever we study such an
alliances we are supposing that the graph contains it.

Notice that the terminology used for alliances provides a very useful tool which
can be used while proving several results, {\em i.e.}, a set of vertices $M$ is a
$k$-monopoly in $G$ if and only if for every vertex $v$ of $G$,
$\delta_M(v)\ge \delta_{\overline{M}}(v)+2k$ (from now we will call this expression
{\em the $k$-monopoly condition}) and we will say that $M$ is a $k$-monopoly in $G$
if and only if every $v$ of $G$ satisfies the $k$-monopoly condition for $M$.

An interesting possible generalization of alliances in graphs (and some other related parameters) is given by \cite{fernau-survey}. In this work is proposed a new framework, which the authors call $(D,O)$-alliances. The main idea of this allows not only to characterize several known variants of alliances, but also suggest a unifying framework for its study. In this sense, a \emph{$(D,O)$-alliance}, with $D,O\subseteq\mathbb{Z}$ in a graph $G=(V,E)$ is a set $S$ such that for any $v\in S$, $\delta_S(v)-\delta_{\overline{S}}(v)\in D$ and for any $v\in N(S)\setminus S$, $\delta_S(v)-\delta_{\overline{S}}(v)\in O$. According to this, it is clear to observe that a defensive $k$-alliance can be understood as a $(\{z\in\mathbb{Z}\,:\,z\geq k\},\mathbb{Z})$-alliance,
and an offensive $k$-alliance as a $(\mathbb{Z},\{z\in\mathbb{Z}\,:\, z\geq k\})$-alliance.

\subsection{Signed (total) domination}

Given a graph $G=(V,E)$ and a function $f: V\rightarrow \{-1,1\}$ we consider the following for $f$:
\begin{itemize}
\item $f$ is a \emph{signed dominating function} for $G$ if $f(N[v])=\sum_{u\in N[v]}f(u)\geq 1$,
for all $v\in V$.
\item $f$ is a \emph{signed total dominating function} for $G$ if $f(N(v))=\sum_{u\in N(v)}f(u)\geq 1$,
for all $v\in V$.
\item $f$ is a {\em signed $k$-dominating function} for $G$ if $f(N[v])\geq k$ for all $v\in V$.
\item $f$ is a {\em signed total $k$-dominating function} for $G$ if $f(N(v))\geq k$ for all $v\in V$.
\end{itemize}
The minimum weight $\sum_{v\in V}f(v)$ of a signed (total) ($k$-dominating) dominating function $f$ is the {\em signed $($total$)$ $(k$-domination$)$ number} of $G$ and they are denoted in the following way.\\

\noindent\begin{tabular}{|c|c|c|c|}
  \hline
  signed domination& signed total domination & signed $k$-domination & signed total $k$-domination \\ \hline
  $\gamma_{s}(G)$ & $\gamma_{st}(G)$ & $\gamma_{s}^k(G)$ & $\gamma_{st}^k(G)$ \\
  \hline
\end{tabular}
\\

Notice that, if $k=1$, then a signed (total) $1$-dominating function is a standard signed (total) dominating function for $G$. Also, any kind of signed (total) ($k$-dominating) dominating function $f$ of $G$ induces two disjoint sets of vertices $B_{1}$ and $B_{-1}$, such that for every
vertex $v\in B_i$, $f(v)=i$ with $i\in \{-1,1\}$. Hereby we will represent such a function $f$ by the sets $B_{1}$ and $B_{-1}$ induced by $f$ and we
will write $f=(B_{1},B_{-1})$. A signed (total) ($k$-dominating) dominating function $f$ of minimum weight is called a $\gamma$-function with
$\gamma\in \{\gamma_s(G),\gamma_{st}(G),\gamma_s^k(G),\gamma_{st}^k(G)\}$, respectively.

\subsection{Connections between concepts}

Observing the definitions of monopoly and alliance we see that both concepts are
closely related. That is, let $M$ be a $0$-monopoly in $G=(V,E)$ and let $v\in V$.
Hence, $v$ has at least half of its neighbors in $M$, {\em i.e.},
$\delta_M(v)\ge \frac{\delta(v)}{2}$, which leads to
$\delta_M(v)\ge \delta_{\overline{M}}(v)$. Since this is satisfied for every
vertex of $G$ we obtain that $M$ is a global defensive $0$-alliance and also
a global offensive $0$-alliance. On the contrary, let $A$ be a global defensive
$0$-alliance which is also a global offensive $0$-alliance in $G$. Hence, for
every vertex $u\in V$ we have that $\delta_A(v)\ge \delta_{\overline{A}}(v)$,
which leads to $\delta_A(v)\ge \frac{\delta(v)}{2}$. Therefore, $A$ is a $0$-monopoly.

\cite{kdaf} defined the concept of global powerful $k$-alliances. Nevertheless,
it was not taken into account the possibility of studying the cases in which a set is
a global defensive $k$-alliance and also a global offensive $k$-alliance. According
to the concept of monopoly we observe the importance of such a case, which is one of
our motivations to develop the present investigation.

We continue with a relationship between signed total domination, alliances and monopolies.

\begin{theorem}\label{sig-tot-monop-alliance}
Let $G=(V,E)$ be a graph and  let $k\in \left\{1,\ldots,\left\lfloor \frac{\delta(G)}{2}\right\rfloor\right\}$
be an integer. The following statements are equivalent:
\begin{enumerate}[{\rm (i)}]
\item $M\subset V$ is a $k$-monopoly in $G$;
\item $M$ is a global defensive $(2k)$-alliance and a global offensive $(2k)$-alliance in $G$;
\item $f=(B_1=M,B_{-1}=\overline{M})$ is a signed total $(2k)$-dominating function for $G$.
\end{enumerate}
Moreover, if $k=0$, then \emph{(i)} and \emph{(ii)} are also equivalent.
\end{theorem}

\begin{proof}
The equivalence between (i) and (ii) is straightforward since for every set of vertices $M$
and every vertex $v$ of $G$, the conditions $\delta_M(v)\ge \frac{\delta(v)}{2}+k$ and
$\delta_M(v)\ge \delta_{\overline{M}}(v)+2k$ are equivalent for every $k\in \left\{0,\ldots,\left\lfloor \frac{\delta(G)}{2}\right\rfloor\right\}$.

Let $M$ be a global defensive $(2k)$-alliance and a global offensive $(2k)$-alliance in $G$.
Let the function $f:V\rightarrow \{-1,1\}$ be such that for any $v\in V$, it follows $f(v)=1$
if $v\in M$ and, $f(v)=-1$ otherwise. 
If $v\in M$, then since $M$ is a global defensive $(2k)$-alliance in $G$, we have that
\begin{align*}
f(N(v))&=f(N_M(v))+f(N_{\overline{M}}(v))\\
&=\delta_M(v)-\delta_{\overline{M}}(v)\\
&\ge \delta_{\overline{M}}(v)+2k-\delta_{\overline{M}}(v)\\
&=2k.
\end{align*}
Now, if $v\in \overline{M}$, then by using that $M$ is a global offensive $(2k)$-alliance in $G$, the same computation as above gives that
$f=(B_1=M,B_{-1}=\overline{M})$ is a signed total $(2k)$-dominating function for $G$.

On the other hand, let $f'=(B'_1,B'_{-1})$ be a signed total $(2k)$-dominating function for $G$. Let $M'=B'_1$ and let the vertex $u\in V$.
If $u\in M'$, then since $f'$ is a signed total $(2k)$-dominating function in $G$, we have that
\begin{align*}
\delta_{M'}(u)&=f'(N_{M'}(u))\\
&=f'(N(u))-f'(N_{\overline{M'}}(u))\\
&\ge 2k-f'(N_{\overline{M'}}(u))\\
&=\delta_{\overline{M'}}(u)+2k.
\end{align*}
Thus, $M'$ is a global defensive $(2k)$-alliance in $G$. Finally, since $f'$ is a signed total $(2k)$-dominating function for $G$, if $u\in \overline{M'}$, then as above we deduce that
$M'$ is a global offensive $(2k)$-alliance.
\end{proof}

The following corollary is a direct consequence of Theorem \ref{sig-tot-monop-alliance} (ii) and (iii). We omit the proof.

\begin{corollary}
Let $G=(V,E)$ be a graph and  let $k\in \left\{1,\ldots,\left\lfloor \frac{\delta(G)}{2}\right\rfloor\right\}$ be an integer. A set $M\subset V$ is a global defensive $k$-alliance and a global offensive $k$-alliance in $G$ if and only if $f=(B_1=M,B_{-1}=\overline{M})$ is a signed total $k$-dominating function for $G$.
\end{corollary}

Now we prove a connection between signed domination and powerful alliances.

\begin{theorem}\label{characterization}
Let $G=(V,E)$ be a graph and let $k\in \{0,\ldots,\delta(G)\}$. Then $S\subset V$ is a global
powerful $k$-alliance in $G$ if and only if $f=(B_1=S,B_{-1}=\overline{S})$ is a signed
$(k+1)$-dominating function for $G$.
\end{theorem}

\begin{proof}
Let $S$ be a global powerful $k$-alliance in $G$. So, $S$ is a global defensive
$k$-alliance and a global offensive $(k+2)$-alliance in $G$. Let $f=(B_1=S,B_{-1}=\overline{S})$
be a function in $G$ and let $v\in V$. We consider the following cases.\\

\noindent \underline{Case 1: $v\in S$.} Since $S$ is a global defensive $k$-alliance in $G$, we have that
\begin{align*}
f(N[v])&=f(N_S(v))+f(N_{\overline{S}}(v))+1\\
&=\delta_S(v)-\delta_{\overline{S}}(v)+1\\
&\ge \delta_{\overline{S}}(v)+k-\delta_{\overline{S}}(v)+1\\
&=k+1.
\end{align*}

\noindent \underline{Case 2: $v\in \overline{S}$.} Since $S$ is a global offensive $(k+2)$-alliance in $G$, we have that
\begin{align*}
f(N[v])&=f(N_S(v))+f(N_{\overline{S}}(v))-1\\
&=\delta_S(v)-\delta_{\overline{S}}(v)-1\\
&\ge \delta_{\overline{S}}(v)+k+2-\delta_{\overline{S}}(v)-1\\
&=k+1.
\end{align*}
Thus, $f=(B_1=S,B_{-1}=\overline{S})$ is a signed $(k+1)$-dominating function for $G$.

On the other hand, let $f'=(B'_1,B'_{-1})$ be a signed $(k+1)$-dominating function in $G$. We will show
that $A=B'_1$ is a global powerful $k$-alliance in $G$. Let $u\in V$. We consider the following.\\

\noindent \underline{Case 3: $u\in A$.} Since $f'$ is a signed $(k+1)$-dominating function for $G$, we have that
\begin{align*}
\delta_A(u)&=f'(N_A(u))\\
&=f'(N[u])-f'(N_{\overline{A}}(u))-1\\
&\ge k+1-f'(N_{\overline{A}}(u))-1\\
&=\delta_{\overline{A}}(u)+k.
\end{align*}
Thus, $A$ is a global defensive $k$-alliance in $G$.\\

\noindent \underline{Case 4: $u\in \overline{A}$.} Since $A$ is a signed $(k+1)$-dominating function in $G$, we have that
\begin{align*}
\delta_A(u)&=f'(N_A(u))\\
&=f'(N[u])-f'(N_{\overline{A}}(u))+1\\
&\ge k+1-f'(N_{\overline{A}}(u))+1\\
&=\delta_{\overline{A}}(u)+k+2
\end{align*}
Thus, $A$ is a global offensive $(k+2)$-alliance and, as a consequence, $A$ is
a global powerful $k$-alliance in $G$. Therefore, the proof is complete.
\end{proof}

\begin{corollary}
For any graph $G$ of order $n$ and any integer $k\in \{0,\ldots,\delta(G)\}$,
$$\gamma_s^{k+1}(G)=2\gamma_k^{p}(G)-n.$$
\end{corollary}

\begin{proof}
Let $S$ be a $\gamma_k^{p}(G)$-set. By Theorem \ref{characterization}, $f=(B_1=S,B_{-1}=\overline{S})$ is a signed total $(k+1)$-dominating function of minimum weight in $G$. Thus $\gamma_s^{k+1}(G)=|S|-|\overline{S}|$. Since $|S|+|\overline{S}|=n$ and $\gamma_k^{p}(G)=|S|$, the result follows by adding these two equalities above.
\end{proof}

According to the above ideas we can resume the relationships which motivated our work in the following table.\\

\begin{center}
\begin{tabular}{|c|c|c|}
  \hline
  \mbox{$k$-monopoly ($k\ge 0$)} & $\Leftrightarrow$ & $\begin{array}{c}
                                                                       \mbox{Global defensive $(2k)$-alliance and} \\
                                                                       \mbox{global offensive $(2k)$-alliance}
                                                                     \end{array}
  $ \\  \hline
  \mbox{$k$-monopoly} ($k\ge 1$) & $\Leftrightarrow$ & \mbox{Signed total $(2k)$-domination} \\  \hline
  \mbox{Signed total $k$-domination ($k\ge 1$)} & $\Leftrightarrow$ & $\begin{array}{c}
                                                                       \mbox{Global defensive $k$-alliance and} \\
                                                                       \mbox{global offensive $k$-alliance}
                                                                     \end{array}
  $ \\  \hline
  \mbox{Signed $(k+1)$-domination ($k\ge 0$)} & $\Leftrightarrow$ & $\begin{array}{c}
                                                                       \mbox{Global defensive $k$-alliance and} \\
                                                                       \mbox{global offensive $(k+2)$-alliance}\\
                                                                       \mbox{(A global powerful $k$-alliance)}
                                                                     \end{array}
  $\\
  \hline
\end{tabular}\\
\end{center}

Notice that the definition of signed (total) $k$-dominating function is restricted to
$k\ge 1$ while $k$-alliances are defined for any $k\in \{-\Delta(G),\ldots,\Delta(G)\}$
and $k$-monopolies can be defined for some integer $k$ whose limits are presented further.
In this sense, these concepts being quite similar between them could be generalized for
$k$ being zero or negative. To obtain a meaningful negative lower bound for $k$-monopolies
we involve another well-known concept: total domination\footnote{A set $D$ is a \emph{total dominating set} in a graph $G$ if every vertex of $G$ is adjacent to a vertex of $D$. The minimum cardinality of a total dominating set is the \emph{total domination number}, denoted by $\gamma_t(G)$.}. Namely, every $k$-monopoly,
$k\geq 0$ is also a total dominating set for $G$. To remain this property also for $k<0$, we
need to demand $\delta_M(v)\geq 1$ for every $v\in V(G)$. Therefore in this work we propose
the following definition of monopolies and we study some of its mathematical properties.

Given a integer $k\in \left\{1-\left\lceil \frac{\delta(G)}{2}\right\rceil,\ldots,\left\lfloor \frac{\delta(G)}{2}\right\rfloor\right\}$
and a set $M$, a vertex $v$ of $G$ is said to be $k$-{\em controlled} by $M$ if $\delta_M(v)\ge \frac{\delta(v)}{2}+k$.
The set $M$ is called a $k$-{\em monopoly} if it $k$-controls every vertex $v$ of $G$.
The minimum cardinality of any $k$-monopoly is the $k$-{\em monopoly number} and it is
denoted by $\mathcal{M}_k(G)$. A monopoly of cardinality $\mathcal{M}_k(G)$ is called
a $\mathcal{M}_k(G)$-set. In particular notice that for a graph with a leaf (vertex of degree one), there
exist only $0$-monopolies and the neighbor of every leaf is in each $\mathcal{M}_0$-set.

Notice that every non trivial graph $G$ contains at least one $k$-monopoly, with
$k\in \left\{1-\left\lceil \frac{\delta(G)}{2}\right\rceil,\right.$ $\left.\ldots,\left\lfloor \frac{\delta(G)}{2}\right\rfloor\right\}$,
since every vertex of $G$ satisfies
the $k$-monopoly condition for the whole vertex set $V(G)$. Also, if $G$ has an isolated
vertex, $\mathcal{M}_k(G)$ does not exists. But if $G$ has no isolated vertices, then, since
$\mathcal{M}_k(G)$-set is also a total dominating set, we have $\mathcal{M}_k(G)\geq 2$.
Thus, we can say that in
general for any graph $G$ of order $n$, $2\le \mathcal{M}_k(G)\le n$.

The last result of this section reveals a connection between
$\mathcal{M}_{1-\left\lceil \frac{\delta(G)}{2}\right\rceil}(G)$ and $\gamma_t(G)$.

\begin{theorem}\label{regular}
For any $r$-regular graph $G$,
$$\mathcal{M}_{1-\left\lceil \frac{r}{2}\right\rceil}(G)=\gamma_t(G).$$
\end{theorem}

\begin{proof} Let $q=\frac{r}{2}+1-\left\lceil \frac{r}{2}\right\rceil$ and let $M$ be a
$\mathcal{M}_{1-\left\lceil \frac{r}{2}\right\rceil}(G)$-set.  If $r$ is even, then
$q=1$ and if $r$ is odd, then $q=\frac{1}{2}$. In both cases, for any vertex $v$ of $G$, $\delta_M(v)\geq 1$,
since $\delta_M(v)$ is an integer. Hence $M$ is a total dominating set and
$\mathcal{M}_{1-\left\lceil \frac{r}{2}\right\rceil}(G)\geq \gamma_t(G)$. If $A$ is a $\gamma_t(G)$-set, then for every vertex
$v\in V$ we obtain $\delta_A(v)\geq 1\geq q$, since $q\in \{\frac{1}{2},1\}$. Thus, $A$ is
also a $(1-\left\lceil \frac{r}{2}\right\rceil)$-monopoly and $\mathcal{M}_{1-\left\lceil \frac{r}{2}\right\rceil}(G)\leq \gamma_t(G)$, which yields the equality.
\end{proof}

\section{Complexity}

Studies about complexity of signed domination were first presented by \cite{heni}. After that \cite{heni2} has shown that signed
total domination problem is NP-complete even restricted to bipartite or chordal
graphs. This last work was continued by \cite{liang}, where the NP-completeness
of signed (total) $k$-domination problem was shown for $k\geq 2$. Consequently, by Theorem \ref{sig-tot-monop-alliance}
the $k$-monopoly problem is also NP-complete for every $k\geq 1$. Hence, it remains to investigate the complexity of $k$-monopolies
for $1-\left\lceil \frac{\delta(G)}{2}\right\rceil\leq k\leq 0$. As mentioned in the introduction,
the complexity and also several inapproximation results are known for a closed monopolies, see \cite{mishra,monopol2,monopol6,monopol3}.

On the other hand, also the global defensive $k$-alliance problem is NP-complete (unpublished manuscript \cite{fernau-PC}) as well as global
offensive $k$-alliance problem (see \cite{fernau-off}), but not both together. Notice
that global powerful $k$-alliance problem is NP-complete as shown by \cite{fernau} but,
as we mention before, a global powerful $k$-alliance is a global defensive $k$-alliance
and a global offensive $(k+2)$-alliance. Here we follow a similar approach as
\cite{heni} to show that $0$-monopoly problem is NP-complete. We will show the
polynomial time reduction on the total domination set problem:\newline\medskip

Problem: \textbf{TOTAL\ DOMINATION\ SET (TDS)}

\textbf{INSTANCE:} A graph $G$ and a positive integer $k\leq |V(G)|$.

\textbf{QUESTION:} Is $\gamma _{t}(G)\leq k$?\newline \medskip

Problem: \textbf{0-MONOPOLY}

\textbf{INSTANCE:} A graph $G$ and a positive integer $k\leq |V(G)|$.

\textbf{QUESTION:} Is $\mathcal{M}_0(G)\leq k$? \\ \medskip

Recall that the total domination set problem is NP-complete even when restricted to
bipartite graphs (see \cite{LaPf}) or to chordal graphs (see
\cite{Pfaf}).\bigskip

\begin{theorem}
Problem 0-MONOPOLY is NP-complete, even when restricted to bipartite or chordal graphs.
\end{theorem}

\begin{proof} It is obvious that $0$-monopoly is a member of NP since for a given set $M$
with $|M|\leq k$ we can check in polynomial time
for each vertex $v$ of a graph $G$ if $v$ is controlled by $M$.

Let $G$ be a graph of order $n$ and size $m$. We construct a graph $H$ from $G$ as follows. For every vertex $v$
add $\delta_G(v)-1$ paths on five vertices and connect $v$ with an
edge to every middle vertex of these paths. Hence to obtain $H$ from $G$ we added
$5\sum_{v\in V(G)}(\delta_G(v)-1)=10m-5n$ vertices and the same amount
of edges. (Notice that we have added exactly $4m-2n$ leaves.) Clearly this can
be done in polynomial time. Also, if $G$ is bipartite or chordal graph, so is
$H$. Next we claim $\mathcal{M}_0(H)=6m-3n+\gamma _{t}(G)$.

To prove this, let $M$ be a $0$-monopoly of $H$. Let $v_{1}v_{2}v_{3}v_{4}v_{5}$
be an arbitrary path added to $G$. Clearly $v_{2},v_{4}\in M$, since they are unique
neighbors of $v_{1}$ and $v_{5}$, respectively. Moreover, if $v_{3}$ is not in $M$,
then both $v_{1}$ and $v_{5}$ must be in $M$ to control $v_{2}$ and
$v_{4}$, respectively. Since $M$ has minimum cardinality, this implies that $v_{3}\in M$. Let $v\in V(G)$.
By the above, $v$ has $\delta_G(v)-1$ neighbors in $M$ outside of $G$. Since $\delta_H(v)=2\delta_G(v)-1$, $v$ needs an additional
neighbor in $M\cap V(G)=P$ to be controlled by $M$. Hence, $P$ forms a total dominating
set of $G$ and so $\gamma _{t}(G)\leq |P|$. Altogether
\begin{eqnarray*}
\mathcal{M}_0(H)=|M|=|P|+3\sum_{v\in V(G)}(\delta_G(v)-1)\geq\gamma _{t}(G)+6m-3n.
\end{eqnarray*}

On the other hand, suppose $S$ is a $\gamma _{t}(G)$-set of $G$. We will show that
$M=S\cup \{v\in V(H)-V(G):\delta_H(v)>1\}$ is a $0$-monopoly for $H$.
Every vertex $v\in V(H)$ with $\delta_H(v)=1$ has a neighbor of degree two which is
in $M$. Without loss of generality, every vertex $v\in V(H)-V(G)$ with $\delta_H(v)=2$ has one neighbor of degree
1 and the other neighbor which is in $M$ and we have
$1=\delta_{M}(v)\geq \delta _{\overline{M}}(v)=1$. Every other vertex $v\in V(H)-V(G)$
has degree three, and two of its neighbors are in $V(H)-V(G)$ with degree two and
thus they are in $M$. Hence $2=\delta _{M}(v)\geq \delta _{\overline{M}}(v)=1$. It remains to
check vertices from $V(G)$. Let $v$ be a vertex
with $\deg _{H}(v)=2\deg _{G}(v)-1$. Since $S$ is a $\gamma _{t}(G)$ set, $v$ has
at least one neighbor in $S$ and additional $\delta_G(v)-1$ vertices in $M
$ in $V(H)-V(G)$. Altogether $v$ has at least $\delta_G(v)$ neighbors in $M$,
which is more than half of its neighbors. The next calculation ends the
proof of the claim:
\begin{eqnarray*}
\mathcal{M}_0(H) &\leq &|M|=|S|+|\{v\in V(H)-V(G):\delta_H(v)>1\}| \\
&=&\gamma _{t}(G)+3\sum_{v\in V(G)}(\delta_G(v)-1) \\
&=&\gamma _{t}(G)+6m-3n.
\end{eqnarray*}

Therefore, we have that if $j=6m-3n+k$, then $\gamma_{t}(G)\leq k$ if and only
if $\mathcal{M}_0(H)\leq j$ and the proof is completed.
\end{proof}

Once having studied the complexity of finding a $0$-monopoly in a graph, it remains to investigate the complexity for $1-\left\lceil \frac{\delta(G)}{2}\right\rceil\leq k\leq -1$, which we leave as an open problem.

\section{Bounding $\mathcal{M}_k(G)$}

In this section we present bounds for $\mathcal{M}_k(G)$ with respect to the minimum and maximum degrees of $G$
and with respect to the order and size. First notice that the $k$-monopoly condition $\delta_{M}(v)\geq \frac{\delta(v)}{2}+k$ is equivalent to the following expressions:
\begin{equation}\label{cond-monopoly-grado1}
\delta_{\overline{M}}(v)\leq \frac{\delta(v)}{2}-k.
\end{equation}

\begin{theorem}\label{general-bounds}
Let $G$ be a graph of order $n$, minimum degree $\delta$ and maximum degree
$\Delta$. Then for any integer
$k\in \left\{1-\left\lceil \frac{\delta(G)}{2}\right\rceil,\ldots,\left\lfloor \frac{\delta(G)}{2}\right\rfloor\right\}$,
$$\left\lceil\frac{\Delta+2k+2}{2}\right\rceil\le \mathcal{M}_k(G)\le n-\left\lfloor\frac{\delta-2k}{2}\right\rfloor.$$
\end{theorem}

\begin{proof}
Let $A$ be a set of vertices of $G$ such that $|A|=\left\lfloor\frac{\delta-2k}{2}\right\rfloor$
and let $v$ be a vertex of $G$. Hence $\delta_A(v)\le \left\lfloor\frac{\delta-2k}{2}\right\rfloor\le \frac{\delta-2k}{2}$. So,
$$\delta_{\overline{A}}(v)\ge \delta(v)-\frac{\delta-2k}{2}\ge\delta(v)-\frac{\delta(v)-2k}{2}=\frac{\delta(v)+2k}{2}.$$
Thus we have $2\delta_{\overline{A}}(v)\ge \delta(v)+2k=\delta_{\overline{A}}(v)+\delta_{A}(v)+2k$,
which leads to $\delta_{\overline{A}}(v)\ge \delta_{A}(v)+2k$. Therefore $\overline{A}$ is a $k$-monopoly in $G$ and the upper bound follows.

On the other hand, let $M$ be a $\mathcal{M}_k(G)$-set and let $u$ be a vertex of maximum degree in $G$. By (\ref{cond-monopoly-grado1}) we have that
$$\Delta =\delta_M(u)+\delta_{\overline{M}}(u)\le \delta_M(u)+\frac{\delta(u)}{2}-k=\delta_M(u)+\frac{\Delta}{2}-k,$$
which leads to $\frac{\Delta}{2}+k\le \delta_M(u)$. Now, if $u\in M$, then we obtain that $\frac{\Delta}{2}+k\le \delta_M(u)\le |M|-1$ and, as a consequence, $\frac{\Delta+2k+2}{2}\le |M|$. Conversely, if $u\notin M$, then $\frac{\Delta}{2}+k\le \delta_M(u)\le |M|$ which leads to $\frac{\Delta+2k}{2}\le |M|$. Therefore, $|M|\ge \displaystyle\max\left\{\frac{\Delta+2k}{2},\frac{\Delta+2k+2}{2}\right\}=\frac{\Delta+2k+2}{2}$ and the lower bound follows.
\end{proof}

As the following corollary shows the above bounds are tight.

\begin{corollary}\label{value-complete}
For every complete graph $K_n$ and every
$k\in \left\{1-\left\lceil \frac{\delta(G)}{2}\right\rceil,\ldots,\left\lfloor \frac{\delta(G)}{2}\right\rfloor\right\}$,
$$\mathcal{M}_k(K_n)=\left\lceil\frac{n+2k+1}{2}\right\rceil.$$
\end{corollary}

\begin{proof}
From Theorem \ref{general-bounds} we have that
$\left\lceil\frac{n+2k+1}{2}\right\rceil\le \mathcal{M}_k(K_n)\le n-\left\lfloor\frac{n-2k-1}{2}\right\rfloor.$
If $n-2k-1$ is even, then $n+2k+1$ is even and we obtain that
$$\left\lceil\frac{n+2k+1}{2}\right\rceil\le\mathcal{M}_k(K_n)\le n-\left\lfloor\frac{n-2k-1}{2}\right\rfloor=n-\frac{n-2k-1}{2}=\frac{n+2k+1}{2}=\left\lceil\frac{n+2k+1}{2}\right\rceil.$$
On the other hand, if $n-2k-1$ is odd, then $n+2k+1$ is odd and we have that
$$\left\lceil\frac{n+2k+1}{2}\right\rceil\le\mathcal{M}_k(K_n)\le n-\left\lfloor\frac{n-2k-1}{2}\right\rfloor=n-\frac{n-2k-2}{2}=\frac{n+2k+2}{2}=\left\lceil\frac{n+2k+1}{2}\right\rceil.$$
\end{proof}

Next we obtain a lower bound for $\mathcal{M}_k(G)$ in terms of order and size of $G$.

\begin{theorem}\label{th-lower-bound-size}
For any graph $G$ of order $n$ and size $m$ and for every
$k\in \left\{1-\left\lceil \frac{\delta(G)}{2}\right\rceil,\ldots,\left\lfloor \frac{\delta(G)}{2}\right\rfloor\right\}-\{0\}$,
$$\mathcal{M}_k(G)\ge \left\lceil\frac{3kn-m}{2k}\right\rceil.$$
\end{theorem}

\begin{proof}
Let $M$ be a $\mathcal{M}_k(G)$-set. Since every vertex $v\in \overline{M}$ satisfies
that $\delta_M(v)\ge \delta_{\overline{M}}(v)+2k\ge 2k$, we have that
$c(M,\overline{M})\ge 2k|\overline{M}|=2k(n-|M|)$, where $c(M,\overline{M})$ is the edge
cut set between $M$ and $\overline{M}$. Since $\delta_M(v)\ge \delta_{\overline{M}}(v)+2k$ holds
for every vertex $v\in M$, we have
\begin{align*}
2k|\overline{M}|&\le c(M,\overline{M})\\
&=\sum_{v\in M}\delta_{\overline{M}}(v)\\
&\le \sum_{v\in M}(\delta_M(v)-2k)\\
&= 2|E(\langle M\rangle)|-2k|M|,
\end{align*}
which leads to $|E(\langle M\rangle)|\ge kn$. 
Since $m\ge |E(\langle M\rangle)|+c(M,\overline{M})$, we obtain that
$m\ge kn+2k(n-|M|)=3kn-2k|M|$ and the result follows.
\end{proof}

To see the tightness of the above bound we consider the following family $\mathcal{F}$ of graphs.
We begin with a complete graph $K_t$ with set of vertices $V=\{v_0,v_1,\ldots,v_{t-1}\}$ and $t-1\equiv 0\, (mod\; 4)$ and $t$ isolated vertices $U=\{u_0,u_1,\ldots,u_{t-1}\}$. From now on all the
operations with subindexes of $v_i$ or $u_i$ are done modulo $t$. To obtain a graph $G\in \mathcal{F}$,
for every $i\in \{0,\ldots,t-1\}$, we add the edges $u_iv_i$, $u_iv_{i+1}$,
$u_iv_{i+2}$, \ldots, $u_iv_{i+(t-3)/2}$. Notice that $G$ has order $2t$ and size $t(t-1)$ and every
vertex $v_i\in V$ has $\frac{t-1}{2}$ neighbors in $U$ and vice versa. Hence $\delta (G)=\frac{t-1}{2}$.
Suppose $k=\left\lfloor \frac{t-1}{4}\right\rfloor$. If $v\in V$, then
$\delta_V(v)=t-1=\frac{t-1}{2}+\frac{t-1}{2}=\delta_U(v)+\frac{t-1}{2}=\delta_U(v)+2k$. Also if $v\in U$, then $\delta_V(v)=\frac{t-1}{2}=\delta_U(v)+\frac{t-1}{2}=\delta_U(v)+2k$. Thus $V$ is a $k$-monopoly in $G$.
By Theorem \ref{th-lower-bound-size} we have $\mathcal{M}_k(G)=t$ and the bound is achieved.

\begin{theorem}
For any $r$-regular graph $G$ of order $n$ and for every
$k\in \left\{1-\left\lceil \frac{\delta(G)}{2}\right\rceil,\ldots,\left\lfloor \frac{\delta(G)}{2}\right\rfloor\right\}$,
$$\mathcal{M}_k(G)\ge \left\lceil\frac{n(2k+r)}{2r}\right\rceil.$$
\end{theorem}

\begin{proof}
Let $V$ be the vertex set of $G$ and let $M$ be a $\mathcal{M}_k(G)$-set. For any vertex $v\in V$ and any $M\subset V$ we have that $\delta(v)=\delta_M(v)+\delta_{\overline{M}}(v)$. By subtracting $2\delta_{\overline{M}}(v)$ in both sides of the equality we obtain  $\delta(v)-2\delta_{\overline{M}}(v)=\delta_M(v)-\delta_{\overline{M}}(v)$. Making a sum for every vertex of $G$ and using the fact that $G$ is $r$-regular, it follows
$$\sum_{v\in V}(\delta_M(v)-\delta_{\overline{M}}(v))=\sum_{v\in V}(\delta(v)-2\delta_{\overline{M}}(v))=nr-2\sum_{v\in V}\delta_{\overline{M}}(v)= nr-2r|\overline{M}|=r|M|-r|\overline{M}|.$$
Thus, $\sum_{v\in V}(\delta_M(v)-\delta_{\overline{M}}(v))=r|M|-r|\overline{M}|$. Since every vertex $v\in V$ satisfies $\delta_M(v)\ge \delta_{\overline{M}}(v)+2k$, we have
$$2kn=\sum_{v\in V}2k\le  \sum_{v\in V}(\delta_M(v)-\delta_{\overline{M}}(v))=r|M|-r|\overline{M}|=2r|M|-rn$$
and the result follows.
\end{proof}

As we will see in Proposition \ref{cycles}, the above bound is tight. For instance, it is achieved for the case of cycles $C_{4t}$ for $k=0$.

\section{Exact values for $\mathcal{M}_{k}(G)$}

As already mentioned, for any graph $G$ of order $n$, $2\le \mathcal{M}_k(G)\le n$. We first characterize
the classes of graphs achieving the limit cases for these bounds.

\begin{proposition}
Let $G$ be a graph of order $n$. Then $\mathcal{M}_k(G)=2$ if and only if
$G$ is isomorphic to $P_2$, $P_3$, $P_4$, $C_3$ or $C_4$. Moreover, $k$ is either 0 or 1.
\end{proposition}

\begin{proof}
If $G$ is isomorphic to $P_2$, $P_3$, $P_4$, $C_3$ or $C_4$, then $\delta(G)\leq 2$, $k\in \{0,1\}$ and
$\mathcal{M}_k(G)=2$. On the contrary, suppose that $\mathcal{M}_k(G)=2$. Let $S=\{u,v\}$ be a $\mathcal{M}_k(G)$-set.
Notice that $u$ and $v$ must be adjacent. So, $\delta_{\overline{S}}(u)\le 1$ and $\delta_{\overline{S}}(v)\le 1$
and $G$ must contain at most four vertices.
Moreover, for every vertex $x\notin \{u,v\}$ it follows
$\delta_{\overline{S}}(x)\le 1$. Thus, $\delta(G)\leq 2$, $k\in \{0,1\}$ and we have the following cases. If
$\delta_{\overline{S}}(u)=0$ and $\delta_{\overline{S}}(v)=0$, then $G$ is isomorphic
to $P_2$. If $\delta_{\overline{S}}(u)=1$ and $\delta_{\overline{S}}(v)=0$ (or vice versa), then $G$
is isomorphic to $P_3$. If $\delta_{\overline{S}}(u)=1$ and $\delta_{\overline{S}}(v)=1$,
then $G$ is isomorphic either to $P_4$, $C_3$ or $C_4$, which completes the proof.
\end{proof}

\begin{proposition}\label{characterization-n}
Let $G$ be a graph of order $n$ and minimum degree $\delta$. Then $\mathcal{M}_k(G)=n$ if and only if $k=\left\lfloor\frac{\delta}{2}\right\rfloor$ and either
\begin{enumerate}[{\em (i)}]
\item $\delta$ is even and every vertex of $G$ is adjacent to a vertex of degree $\delta$ or $\delta+1$, or
\item $\delta$ is odd and every vertex of $G$ is adjacent to a vertex of degree $\delta$.
\end{enumerate}
\end{proposition}

\begin{proof}
Suppose $\mathcal{M}_k(G)=n$. Hence, for any vertex $v\in V(G)$, $M=V(G)-\{v\}$ is not $k$-monopoly in $G$. Thus the vertex $v$ or some vertex $u\in N(v)$ does not satisfy the monopoly condition. If $\delta_M(v)<\delta_{\overline{M}}(v)+2k$, then we have that $\delta(v)<2k\le \delta$, a contradiction. Thus $\delta_M(u)<\delta_{\overline{M}}(u)+2k$, which leads to $\delta(u)-1<1+2k$. So $\delta(u)\le 2k+1$. As a consequence, we obtain that $k\ge \frac{\delta-1}{2}$ (or equivalently $k\ge \left\lceil\frac{\delta-1}{2}\right\rceil$). Since $k\le\left\lfloor\frac{\delta}{2}\right\rfloor$, we obtain that $k=\left\lfloor\frac{\delta}{2}\right\rfloor=\left\lceil\frac{\delta-1}{2}\right\rceil$. Thus, $\delta(u)\le 2k+1=2\left\lceil\frac{\delta-1}{2}\right\rceil+1$. Hence, if $\delta$ is even, then we have that $\delta(u)\le \delta+1$, and if $\delta$ is odd, then we have that $\delta(u)\le \delta$. Therefore, (i) and (ii) follow.

On the other hand, suppose $k=\left\lfloor\frac{\delta}{2}\right\rfloor$. Assume $\delta$ is even and every vertex of $G$ is
adjacent to a vertex of degree $\delta$ or $\delta+1$. Hence, let $M\subset V(G)$, let $x\notin M$ and let $u\in N(x)$ having degree $\delta$ or $\delta+1$. So we have,
$$\delta_M(u)\le\delta< 2\left\lfloor\frac{\delta}{2}\right\rfloor+1=2k+1\le \delta_{\overline{M}}(u)+2k.$$
Thus, $M$ is not a $k$-monopoly.

Now, suppose $\delta$ is odd and every vertex of $G$ is adjacent to a vertex of degree $\delta$. As above let $M'\subset V(G)$, let $x'\notin M'$ and let $u'\in N(x')$ having degree $\delta$. So we have,
$$\delta_{M'}(u')<\delta=2\left\lfloor\frac{\delta}{2}\right\rfloor+1=2k+1\le \delta_{\overline{M'}}(u')+2k.$$
Thus, $M'$ is not a $k$-monopoly.

Therefore, any proper subset of $V(G)$ is not a $k$-monopoly and we have that $\mathcal{M}_k(G)=n$.
\end{proof}

The wheel graph of order $n$ is defined as $W_{1,n-1}=K_1+C_{n-1}$, where $+$ represents the join of mentioned graphs. The fan graph $F_{1,n-1}$ of order $n$ is defined as the graph $K_1+P_{n-1}$.

\begin{corollary}$\;$\label{coro-values-n}
\begin{enumerate}[{\em (i)}]
\item For any $r$-regular graph $G$ of order $n$, $\mathcal{M}_{\left\lfloor\frac{r}{2}\right\rfloor}(G)=n$.
\item For any wheel graph $W_{1,n-1}$, $\mathcal{M}_1(W_{1,n-1})=n$.
\item For any fan graph $F_{1,n-1}$, $\mathcal{M}_1(F_{1,n-1})=n$.
\item For any bipartite graph $K_{r,r+1}$, $r$ even, $\mathcal{M}_{\left\lfloor\frac{r}{2}\right\rfloor}(K_{r,r+1})=2r+1$.
\end{enumerate}
\end{corollary}

We continue this section by obtaining exact values for some graph classes. Recall that, by Corollary \ref{value-complete}, for
$k\in \left\{1-\left\lceil \frac{\delta(G)}{2}\right\rceil,\ldots,\left\lfloor \frac{\delta(G)}{2}\right\rfloor\right\}$ we have $\mathcal{M}_k(K_n)=\left\lceil\frac{n+2k+1}{2}\right\rceil$. We continue with complete bipartite graphs.

\begin{proposition}
For every complete bipartite graph $K_{r,t}$ and every
$k\in \left\{1-\left\lceil \frac{\delta(G)}{2}\right\rceil,\ldots,\left\lfloor \frac{\delta(G)}{2}\right\rfloor\right\}$,
$$\mathcal{M}_k(K_{r,t})=\left\lceil\frac{r+2k}{2}\right\rceil+\left\lceil\frac{t+2k}{2}\right\rceil.$$
\end{proposition}

\begin{proof}
Let $X$ and $Y$ be the partition sets of $K_{r,t}$ such that $|X|=r$ and $|Y|=t$ and let $S$ be a subset of
vertices of $K_{r,t}$ such that $|S\cap X|=\left\lceil\frac{r+2k}{2}\right\rceil$ and
$|S\cap Y|=\left\lceil\frac{t+2k}{2}\right\rceil$. Let $v$ be a vertex of $K_{r,t}$. If $v\in X$, then
$$\delta_S(v)=\left\lceil\frac{t+2k}{2}\right\rceil\ge\frac{t+2k}{2} =t+2k-\frac{t+2k}{2}
\ge t-\left\lceil\frac{t+2k}{2}\right\rceil+2k=\delta_{\overline{S}}(v)+2k.$$
Analogously, if $v\in Y$, then we obtain that $\delta_S(v)\ge \delta_{\overline{S}}(v)+2k$. Thus, $S$ is a $k$-monopoly in $K_{r,t}$ and we have that $\mathcal{M}_k(K_{r,t})\le \left\lceil\frac{r+2k}{2}\right\rceil+\left\lceil\frac{t+2k}{2}\right\rceil.$

Now, let $M$ be a $\mathcal{M}_k(K_{r,t})$-set and let $u$ be a vertex of $K_{r,t}$. If $u\in X$, then we have that $\delta_M(u)\ge \delta_{\overline{M}}(u)+2k=t-\delta_M(u)+2k$, which leads to $\delta_M(u)\ge \frac{t+2k}{2}$ and, as a consequence, $|Y\cap M|=\delta_M(u)\ge \left\lceil\frac{t+2k}{2}\right\rceil$. Analogously, if $u\in Y$, then we obtain that $|X\cap M|\ge \left\lceil\frac{r+2k}{2}\right\rceil$. Thus, $\mathcal{M}_k(K_{r,t})=|M\cap X|+|M\cap Y|\ge \left\lceil\frac{r+2k}{2}\right\rceil+\left\lceil\frac{t+2k}{2}\right\rceil$ and the proof is complete.
\end{proof}

Next we study $k$-monopolies of cycles and paths. First notice that the case $k=1$ for cycles follows directly from Corollary \ref{coro-values-n} (i), that is, $\mathcal{M}_1(C_n)=n$.

\begin{proposition}\label{cycles}
For every integer $n\ge 3$,
$$\mathcal{M}_0(C_n)=\mathcal{M}_0(P_n)=\left\{\begin{array}{ll}
                              \frac{n}{2} & \mbox{ if $n\equiv 0$ mod 4,} \\
                                &  \\
                              \frac{n+2}{2} & \mbox{ if $n\equiv 2$ mod 4,} \\
                                &  \\
                              \frac{n+1}{2} & \mbox{ if $n\equiv 1$ mod 4 or  $n\equiv 3$ mod 4.}
                            \end{array}
\right.$$
\end{proposition}

\begin{proof}
By Theorem \ref{regular}, $\mathcal{M}_0(C_n)=\gamma_t(C_n)$ and it is known from \cite{heni5} that
$\gamma_t(C_n)=\left\lfloor \frac{n}{2}\right\rfloor+\left\lceil \frac{n}{4}\right\rceil-\left\lfloor \frac{n}{4}\right\rfloor$. Hence
we are done with cycles.

Let $V(P_n)=\{v_0,\ldots ,v_{n-1}\}$. We proceed by induction on $k\ge1$ where $n=4k+i$ and $i\in\{-1,0,1,2\}$.
Let $M_n$ be a subset of $V(P_n)$ defined as follows.
\begin{itemize}
\item If $n\equiv 0$ (mod 4), then $M_n=\{v_1,v_2,v_5,v_6,\ldots,v_{n-3},v_{n-2}\}$.
\item If $n\equiv 1$ (mod 4), then $M_n=\{v_1,v_2,v_3,v_6,v_7,v_{10},v_{11},\ldots,v_{n-3},v_{n-2}\}$.
\item If $n\equiv 2$ (mod 4), then $M_n=\{v_0,v_1,v_3,v_4,v_7,v_8,v_{11},v_{12},\ldots,v_{n-3},v_{n-2}\}$.
\item If $n\equiv 3$ (mod 4), then $M_n=\{v_0,v_1,v_4,v_5,\ldots,v_{n-3},v_{n-2}\}$.
\end{itemize}

It is straightforward to check that $M_n$ is a $\mathcal{M}_0(P_n)$-set for $k=1$. Notice that $M_4$ is
the unique $\mathcal{M}_0(P_4)$-set. Let $k>1$.
Set $M_{4(k-1)+i}$ is a $\mathcal{M}_0(P_{4(k-1)+i})$-set by induction hypothesis. Clearly, any $0$-monopoly
$M'$ of $P_n$ contains at least two vertices of the last three vertices $v_{n-3},v_{n-2},v_{n-1}$. Hence, these two vertices have no influence
on the vertices of $M'$ from the first $4(k-1)+i$ vertices of the path $P_{4k+i}$. Therefore
$|M'\cap \{v_0,\ldots ,v_{4(k-1)+i}\}|\geq |M_{4(k-1)+i}|$ and $M_{4k+i}=M_{4(k-1)+i}\cup \{v_{n-3},v_{n-2}\}$
is a $\mathcal{M}_0(P_{4k+i})$-set. It is easy to see that $|M_{4k+i}|$ gives the desired values.
\end{proof}

\section{Partitions into $k$-monopolies}

In this section we present some results about partitioning a graphs into monopolies. To this end, we say that a graph $G=(V,E)$ is $k$-monopoly partitionable if there exists a vertex partition $\Pi=\{S_1,\ldots,S_r\}$ of $V$, $r\ge 2$, such that for every $i\in \{1,\ldots,r\}$, $S_i$ is a $k$-monopoly in $G$.

\begin{theorem} \label{partition}
If a graph $G$ is $k$-monopoly partitionable, for some $k\in\{1-\left\lceil \frac{\delta(G)}{2}\right\rceil,\ldots,\left\lfloor \frac{\delta(G)}{2}\right\rfloor\}$, then $r\le 2-2k$ and $k\le 0$.
\end{theorem}

\begin{proof}
Let $S_i,S_j\in \Pi$ and let $v$ be a vertex of $G$. Then we have that
\begin{align*}
\delta_{S_i}(v)&\ge \delta_{\overline{S_i}}(v)+2k\\
&=2k+\sum_{\ell=1,\ell\ne i}^r\delta_{S_{\ell}}(v)\\
&=\delta_{S_j}(v)+2k+\sum_{\ell=1,\ell\ne i,j}^r\delta_{S_{\ell}}(v)\\
&\ge \delta_{\overline{S_j}}(v)+4k+\sum_{\ell=1,\ell\ne i,j}^r\delta_{S_{\ell}}(v)
\end{align*}
Since for every $u$ of $G$, $\delta_{S_{\ell}}(u)\ge 1$ for every $\ell\in \{1,\ldots,r\}$, we
obtain that $\displaystyle\sum_{\ell=1,\ell\ne i,j}^r\delta_{S_{\ell}}(u)\ge r-2$.
So,
\begin{align*}
\delta_{S_i}(v)&\ge \delta_{\overline{S_j}}(v)+4k+r-2\\
&=4k+r-2+\sum_{\ell=1,\ell\ne j}^r\delta_{S_{\ell}}(v)\\
&=\delta_{S_i}(v)+4k+r-2+\sum_{\ell=1,\ell\ne i,j}^r\delta_{S_{\ell}}(v)\\
&\ge \delta_{S_i}(v)+4k+2r-4.
\end{align*}
Thus $2k+r-2\le 0$, which leads to $r\le 2-2k$ and $k\le 1-r/2$. Since $r\ge 2$, we have that $k\le 0$.
\end{proof}

From the above result we have that $G$ can be only partitioned into at most $2-2k$ $k$-monopolies for $k\le 0$. The particular case $k=0$ is next studied. Notice that for instance, cycles of order $4t$ and hypercubes $Q_{2t}$ with $t\ge 1$ are examples of graphs having a partition into two $0$-monopolies.

\begin{proposition}
Let $G$ be a graph having a vertex partition into two $0$-monopolies $\{X,Y\}$. Then the following assertions are satisfied.
\begin{enumerate}[{\em (i)}]
\item For every vertex $v$ of $G$, $\delta_X(v)=\delta_Y(v)$.
\item For every vertex $v$ of $G$, $\delta(v)$ is an even number.
\item The size $m_X$ of $\langle X\rangle$ equals the size $m_Y$ of $\langle Y\rangle$.
\item The cardinality of the edge cut set $c(X,Y)$ produced by the vertex partition $\{X,Y\}$ equals the size $m$ of $G$ minus two times the size of $\langle X\rangle$.
\end{enumerate}
\end{proposition}

\begin{proof}
For every vertex $v$ of $G$ we have that $\delta_X(v)\geq \delta_{Y}(v)$ and $\delta_Y(v)\geq \delta_{X}(v)$. Thus, (i) follows. Now, (ii) follows from the fact that $\delta(v)=\delta_X(v)+\delta_Y(v)=2\delta_X(v)=2\delta_Y(v)$. To prove (iii) we consider the following
$$\sum_{v\in X}\delta_X(v)+\sum_{v\in Y}\delta_X(v)=\sum_{v\in X}\delta_Y(v)+\sum_{v\in Y}\delta_Y(v).$$
Since, $\sum_{v\in Y}\delta_X(v)=\sum_{v\in X}\delta_Y(v)$ we have the result. As a consequence, $m=c(X,Y)+m_X+m_Y$ and by (iii) we obtain (iv).
\end{proof}

A natural question which now arises concerning the computational complexity on the existence of such partitions mentioned above. That is for instance, given a graph $G$, can we decide whether $G$ is $k$-monopoly partitionable? Moreover, if the answer is positive, can we find such partitions by using some efficient algorithm?


\nocite{*}
\bibliographystyle{abbrvnat}


\bibliography{monopolies-1}

\begin{thebibliography}{29}
\providecommand{\natexlab}[1]{#1}
\providecommand{\url}[1]{\texttt{#1}}
\expandafter\ifx\csname urlstyle\endcsname\relax
  \providecommand{\doi}[1]{doi: #1}\else
  \providecommand{\doi}{doi: \begingroup \urlstyle{rm}\Url}\fi

\bibitem[Bermond et~al.(2003)Bermond, Bond, Peleg, and
  Perennes]{small-coalition}
J.-C. Bermond, J.~Bond, D.~Peleg, and S.~Perennes.
\newblock The power of small coalitions in graphs.
\newblock \emph{Discrete Applied Mathematics}, 127\penalty0 (3):\penalty0
  399--414, 2003.
\newblock \doi{10.1016/S0166-218X(02)00241-X}.
\newblock URL
  \url{http://www.scopus.com/inward/record.url?eid=2-s2.0-84867959143&partnerID=40&md5=df993bddbdf24f26b0f2554323bce429}.

\bibitem[Dunbar et~al.(1995)Dunbar, Hedetniemi, Henning, and
  Slater]{signed-one}
J.~Dunbar, S.~Hedetniemi, M.~Henning, and P.~Slater.
\newblock Signed domination in graphs.
\newblock \emph{Graph Theory, Combinatorics, and Applications}, pages 311--322,
  1995.

\bibitem[Dwork et~al.(1988)Dwork, Peleg, Pippenger, and Upfal]{consensus}
C.~Dwork, D.~Peleg, N.~Pippenger, and E.~Upfal.
\newblock Fault tolerance in networks of bounded degree.
\newblock \emph{SIAM Journal on Computing}, 17\penalty0 (5):\penalty0 975--988,
  1988.
\newblock URL
  \url{http://www.scopus.com/inward/record.url?eid=2-s2.0-0024091299&partnerID=40&md5=009c11dd0d050861b69499d3fbf93566}.

\bibitem[Favaron et~al.(2004)Favaron, Fricke, Goddard, Hedetniemi, Hedetniemi,
  Kristiansen, Laskar, and Skaggs]{fava}
O.~Favaron, G.~Fricke, W.~Goddard, S.~Hedetniemi, S.~Hedetniemi,
  P.~Kristiansen, R.~C. Laskar, and R.~Skaggs.
\newblock Offensive alliances in graphs.
\newblock \emph{Discussiones Mathematicae Graph Theory}, 24:\penalty0 263--275,
  2004.

\bibitem[Fernau(2013)]{fernau-PC}
H.~Fernau.
\newblock Private communication.
\newblock 2013.

\bibitem[Fernau and Rodríguez-Vel\'azquez(2014)]{fernau-survey}
H.~Fernau and J.~Rodríguez-Vel\'azquez.
\newblock A survey on alliances and related parameters in graphs.
\newblock \emph{Electronic Journal of Graph Theory and Applications},
  2\penalty0 (1):\penalty0 70--86, 2014.

\bibitem[Fernau et~al.(2009)Fernau, Rodríguez-Vel\'azquez, and
  Sigarreta]{fernau-off}
H.~Fernau, J.~Rodríguez-Vel\'azquez, and J.~Sigarreta.
\newblock Offensive r-alliances in graphs.
\newblock \emph{Discrete Applied Mathematics}, 157\penalty0 (1):\penalty0
  177--182, 2009.
\newblock \doi{10.1016/j.dam.2008.06.001}.
\newblock URL
  \url{http://www.scopus.com/inward/record.url?eid=2-s2.0-56649115970&partnerID=40&md5=ec6162e86b28fe3af48f2f1c64169a75}.

\bibitem[Fernau et~al.(to appear)Fernau, Rodríguez-Vel\'azquez, and
  Sigarreta]{fernau}
H.~Fernau, J.~Rodríguez-Vel\'azquez, and J.~Sigarreta.
\newblock Global powerful r-alliances and total k-domination in graphs.
\newblock \emph{Utilitas Mathematica}, to appear.

\bibitem[Flocchini et~al.(2003)Flocchini, Královič, Ružička, Roncato, and
  Santoro]{monopol1}
P.~Flocchini, R.~Královič, P.~Ružička, A.~Roncato, and N.~Santoro.
\newblock On time versus size for monotone dynamic monopolies in regular
  topologies.
\newblock \emph{Journal of Discrete Algorithms}, 1\penalty0 (2):\penalty0
  129--150, 2003.
\newblock \doi{10.1016/S1570-8667(03)00022-4}.
\newblock URL
  \url{http://www.scopus.com/inward/record.url?eid=2-s2.0-77952569578&partnerID=40&md5=bc3c283a056a43366c34cdffad9c85e3}.

\bibitem[Garcia-Molina and Barbara(1985)]{voting}
H.~Garcia-Molina and D.~Barbara.
\newblock How to assign votes in a distributed system.
\newblock \emph{Journal of the ACM}, 32\penalty0 (4):\penalty0 841--860, 1985.
\newblock \doi{10.1145/4221.4223}.
\newblock URL
  \url{http://www.scopus.com/inward/record.url?eid=2-s2.0-0022145769&partnerID=40&md5=19a2ca07b3b1f38508e2926817239acd}.

\bibitem[Hattingh et~al.(1995)Hattingh, Henning, and Slater]{heni}
J.~Hattingh, M.~Henning, and P.~Slater.
\newblock The algorithmic complexity of signed domination in graphs.
\newblock \emph{Australasian Journal of Combinatorics}, 12:\penalty0 101--112,
  1995.

\bibitem[Haynes et~al.(2003)Haynes, Hedetniemi, and
  Henning]{GlobalalliancesOne}
T.~Haynes, S.~Hedetniemi, and M.~Henning.
\newblock Global defensive alliances in graphs.
\newblock \emph{Electronic Journal of Combinatorics}, 10\penalty0 (1 R), 2003.
\newblock URL
  \url{http://www.scopus.com/inward/record.url?eid=2-s2.0-3042619803&partnerID=40&md5=de525879814c4abbf82b4b7d27f52fa4}.

\bibitem[Henning(2000)]{heni5}
M.~Henning.
\newblock Graphs with large total domination number.
\newblock \emph{Journal of Graph Theory}, 35\penalty0 (1):\penalty0 21--45,
  2000.
\newblock URL
  \url{http://www.scopus.com/inward/record.url?eid=2-s2.0-23044518657&partnerID=40&md5=eda9769f06132678c59dd3470b350fe7}.

\bibitem[Henning(2004)]{heni2}
M.~Henning.
\newblock Signed total domination in graphs.
\newblock \emph{Discrete Mathematics}, 278\penalty0 (1-3):\penalty0 109--125,
  2004.
\newblock \doi{10.1016/j.disc.2003.06.002}.
\newblock URL
  \url{http://www.scopus.com/inward/record.url?eid=2-s2.0-1142263889&partnerID=40&md5=20dd482bbe17d6173990b948169f1bbe}.

\bibitem[Khoshkhah et~al.(2013)Khoshkhah, Nemati, Soltani, and Zaker]{khoskah}
K.~Khoshkhah, M.~Nemati, H.~Soltani, and M.~Zaker.
\newblock A study of monopolies in graphs.
\newblock \emph{Graphs and Combinatorics}, 29\penalty0 (5):\penalty0
  1417--1427, 2013.
\newblock \doi{10.1007/s00373-012-1214-7}.
\newblock URL
  \url{http://www.scopus.com/inward/record.url?eid=2-s2.0-84882903032&partnerID=40&md5=9962cc0de8828815ac6749b93652024d}.

\bibitem[Kristiansen et~al.(2004)Kristiansen, Hedetniemi, and
  Hedetniemi]{alliancesOne}
P.~Kristiansen, S.~Hedetniemi, and S.~Hedetniemi.
\newblock Alliances in graphs.
\newblock \emph{Journal of Combinatorial Mathematics and Combinatorial
  Computing}, 48:\penalty0 157--177, 2004.

\bibitem[Laskar and Pfaff(1984)]{LaPf}
R.~Laskar and J.~Pfaff.
\newblock Domination and irredundance in split graphs.
\newblock \emph{Technical Report, Department of Mathematical Sciences, Clemson
  University}, 10\penalty0 (1):\penalty0 49--60, 1984.
\newblock \doi{10.1016/j.jda.2011.12.019}.
\newblock URL
  \url{http://www.scopus.com/inward/record.url?eid=2-s2.0-84855548350&partnerID=40&md5=468858cae7328c688f73b0e203b00ac6}.

\bibitem[Liang(2014)]{liang}
H.~Liang.
\newblock On the signed (total) $k$-domination number of a graph.
\newblock \emph{Journal of Combinatorial Mathematics and Combinatorial
  Computing}, 89:\penalty0 87--99, 2014.

\bibitem[Linial et~al.(1993)Linial, Peleg, Rabinovich, and Saks]{monopoly-one}
N.~Linial, D.~Peleg, Y.~Rabinovich, and M.~Saks.
\newblock Sphere packing and local majorities in graphs.
\newblock \emph{\em $2^{nd}$ Israel Symposium on Theory and Computing Systems},
  pages 141--149, 1993.

\bibitem[Mishra(2012)]{mishra}
S.~Mishra.
\newblock Complexity of majority monopoly and signed domination problems.
\newblock \emph{Journal of Discrete Algorithms}, 10\penalty0 (1):\penalty0
  49--60, 2012.
\newblock \doi{10.1016/j.jda.2011.12.019}.
\newblock URL
  \url{http://www.scopus.com/inward/record.url?eid=2-s2.0-84855548350&partnerID=40&md5=468858cae7328c688f73b0e203b00ac6}.

\bibitem[Mishra and Rao(2006)]{monopol2}
S.~Mishra and S.~Rao.
\newblock Minimum monopoly in regular and tree graphs.
\newblock \emph{Discrete Mathematics}, 306\penalty0 (14):\penalty0 1586--1594,
  2006.
\newblock \doi{10.1016/j.disc.2005.06.036}.
\newblock URL
  \url{http://www.scopus.com/inward/record.url?eid=2-s2.0-33745656404&partnerID=40&md5=0f2c2e220120fbe991066315ad65cdc6}.

\bibitem[Mishra et~al.(2002)Mishra, Radhakrishnan, and
  Sivasubramanian]{monopol6}
S.~Mishra, J.~Radhakrishnan, and S.~Sivasubramanian.
\newblock On the hardness of approximating minimum monopoly problems.
\newblock \emph{Lecture Notes in Computer Science (including subseries Lecture
  Notes in Artificial Intelligence and Lecture Notes in Bioinformatics)}, 2556
  LNCS:\penalty0 277--288, 2002.
\newblock URL
  \url{http://www.scopus.com/inward/record.url?eid=2-s2.0-80052395587&partnerID=40&md5=cb24b2adb61a5dbff36910d9addba94c}.

\bibitem[Peleg(2002)]{monopol3}
D.~Peleg.
\newblock Local majorities, coalitions and monopolies in graphs: A review.
\newblock \emph{Theoretical Computer Science}, 282\penalty0 (2):\penalty0
  231--257, 2002.
\newblock \doi{10.1016/S0304-3975(01)00055-X}.
\newblock URL
  \url{http://www.scopus.com/inward/record.url?eid=2-s2.0-0037054318&partnerID=40&md5=32b9c2596729f62832b1564d367adccc}.

\bibitem[Pfaff(1984)]{Pfaf}
J.~Pfaff.
\newblock \emph{Algorithmic complexities of domination-related graph
  parameters}.
\newblock Ph.D. Thesis, Clemson University, 1984.

\bibitem[Shafique and Dutton(2003)]{kdaf}
K.~Shafique and R.~Dutton.
\newblock Maximum alliance-free and minimum alliance-cover sets.
\newblock \emph{Congressus Numerantium}, 162:\penalty0 139--146, 2003.

\bibitem[Shafique and Dutton(2006)]{kdaf1}
K.~Shafique and R.~Dutton.
\newblock A tight bound on the cardinalities of maximun alliance-free and
  minimun alliance-cover sets.
\newblock \emph{Journal of Combinatorial Mathematics and Combinatorial
  Computing}, 56:\penalty0 139--145, 2006.

\bibitem[Sullivan(1986)]{diagnosis}
G.~Sullivan.
\newblock \emph{The complexity of system-level fault diagnosis and
  diagnosability}.
\newblock Ph. D. Thesis, Yale University New Haven. CT, 1986.

\bibitem[Zaker(2012)]{monopol4}
M.~Zaker.
\newblock On dynamic monopolies of graphs with general thresholds.
\newblock \emph{Discrete Mathematics}, 312\penalty0 (6):\penalty0 1136--1143,
  2012.
\newblock \doi{10.1016/j.disc.2011.11.038}.
\newblock URL
  \url{http://www.scopus.com/inward/record.url?eid=2-s2.0-84155191040&partnerID=40&md5=c463fb9151e8e07ebad6dc45643a750e}.

\bibitem[Zelinka(2001)]{signed-total}
B.~Zelinka.
\newblock Signed total domination number of a graph.
\newblock \emph{Czechoslovak Mathematical Journal}, 51\penalty0 (2):\penalty0
  225--229, 2001.
\newblock URL
  \url{http://www.scopus.com/inward/record.url?eid=2-s2.0-23044528767&partnerID=40&md5=9024d6b751af95d67c4254a2d6a091e6}.

\end{thebibliography}

\end{document}